\documentclass[12pt]{amsart}
\usepackage{amsmath,amsfonts,amssymb,amsthm,amstext,pgf,graphicx,verbatim,lmodern,textcomp,color,young,tikz}
\usetikzlibrary{decorations}
\usepackage[mathscr]{euscript}
\usetikzlibrary{decorations.markings}
\usetikzlibrary{arrows}
\usepackage{float}
\usetikzlibrary{arrows.meta, decorations.markings, positioning, shapes}
\usetikzlibrary{arrows.meta}  
\usetikzlibrary{arrows.meta, positioning}
\usepackage{color}
\usepackage{comment}
\usepackage{tikz}
\usepackage{float}
\usepackage{tikz,pgfplots}
\usepackage{enumerate}
\usepackage{hyperref}
\theoremstyle{plain}
\newtheorem{theorem}{Theorem}[section]
\newtheorem{lemma}[theorem]{Lemma}

\theoremstyle{remark}
\newtheorem{remark}[theorem]{Remark}

\theoremstyle{definition}

\newtheorem{example}{Example}[section]

\title{}
\begin{document}
\title [A combinatorial proof of the trace Cayley-Hamilton theorem]{A combinatorial proof of the trace Cayley-Hamilton theorem}
\author[Sudip Bera]{Sudip Bera}
\address[Sudip Bera]{Faculty of Mathematics, Dhirubhai Ambani University, Gandhinagar, India}
\email{sudip\_bera@dau.ac.in}

\begin{abstract}
The deep interconnection between linear algebra and graph theory allows one to interpret classical matrix invariants through combinatorial structures.  
To each $n \times n$ matrix $A$ over a commutative ring $\mathbb{K}$, one can associate a weighted directed graph $\mathcal{D}(A)$, where the algebraic behavior of $A$ is reflected in the combinatorial properties of $\mathcal{D}(A)$.  
In particular, the determinant and characteristic polynomial of $A$ admit elegant formulations in terms of sign-weighted sums over linear subdigraphs of $\mathcal{D}(A)$, thereby providing a graphical interpretation of fundamental algebraic quantities.
Building upon this correspondence, we establish a combinatorial proof of the \emph{Trace Cayley–Hamilton theorem}.  
This theorem furnishes explicit trace identities linking the coefficients of the characteristic polynomial of $A$ with the traces of its successive powers.  
Precisely, if 
\[
p_A(\lambda) = \lambda^n + d_1 \lambda^{n-1} + \cdots + d_n
\]
is the characteristic polynomial of $A$, then for every integer $r \ge 1$, the traces of powers of $A$ satisfy
\[
\text{Tr}(A^r) + \text{Tr}(A^{r-1})d_1 + \cdots + \text{Tr}(A^{r-(n-1)})d_{n-1}+
\begin{cases}
\text{Tr}(A^{r-n})d_n = 0, & r > n,\\[4pt]
r d_r = 0, & 1 \le r \le n.
\end{cases}
\]


		\medskip
		
			
			\medskip  
			
			
\end{abstract}

\maketitle

\section{Introduction}
\label{sec:introduction}
Let $\mathbb{K}$ be a commutative ring, and let $A \in \mathbb{K}^{n \times n}$ be a square matrix of order $n$.  
The characteristic polynomial of $A$ is defined by
\begin{equation}\label{Eqn:c-h-thm,coeef-as-d}
p_A(\lambda) = \lambda^n + d_1 \lambda^{n-1} + d_2 \lambda^{n-2} + \cdots + d_{n-k} \lambda^{k} + \cdots + d_n.
\end{equation}
The classical \emph{Cayley-Hamilton theorem} states that
\[
p_A(A) = A^n + d_1 A^{n-1} + d_2 A^{n-2} + \cdots + d_{n-k} A^{k} + \cdots + d_n I_n = O,
\]
where $O$ denotes the zero matrix.  
Various combinatorial proofs of such fundamental results are well known in the literature~\cite{Comb-Newton-Girad-Disc_Mukherjee-Bera,Straubing1983,11,22}.

\medskip

A less conventional but equally elegant corollary, referred to as the \emph{Trace Cayley-Hamilton theorem}, asserts that
\begin{equation}\label{Eqn:TraceCH}
\operatorname{Tr}(A^r) + \operatorname{Tr}(A^{r-1}) d_1 + \cdots + \operatorname{Tr}(A^{r-(n-1)}) d_{n-1} +
\begin{cases}
\operatorname{Tr}(A^{r-n}) d_n = 0, & r > n,\\[4pt]
r\, d_r = 0, & 1 \le r \le n.
\end{cases}
\end{equation}
When $r \ge n$, this identity follows immediately from the Cayley--Hamilton theorem by multiplying both sides of $p_A(A) = O$ by $A^{r-n}$ and then taking traces.  
However, for arbitrary $r$, no direct algebraic derivation is known.  
In this paper, we present a combinatorial proof of the Trace Cayley--Hamilton theorem using the framework of \emph{weighted directed graphs}.

\smallskip

To every matrix $A = (a_{ij}) \in \mathbb{K}^{n \times n}$, we associate a weighted directed graph $\mathcal{D}(A)$ with vertex set $[n] = \{1, 2, \ldots, n\}$.  
For each ordered pair $(i,j)$, there is a directed edge from $i$ to $j$ with weight $a_{ij}$.  
We briefly recall a few relevant graph-theoretic notions; see~\cite{Brualdi-Cvetkovic-combinatorial-matrix-thy} for further details.

A \emph{linear subdigraph} $\gamma$ of $\mathcal{D}(A)$ is a collection of vertex-disjoint directed cycles.  
A loop, i.e., a cycle of length one, is regarded as a cycle centered at a single vertex.  
The \emph{weight} of $\gamma$, denoted $w(\gamma)$, is the product of the weights of all edges belonging to $\gamma$.  
The number of cycles in $\gamma$ is denoted by $c(\gamma)$.  
For illustration, all linear subdigraphs of the digraph in Figure~\ref{fig:D(A)} are displayed in Figures~\ref{fig:lsd-1} and~\ref{fig:lsd-2}.

\medskip

The cycle decomposition of permutations yields the following elegant expression for the determinant of $A$:
\begin{equation}\label{Eqn:LSD Defn of Det}
\det(A) = \sum_{\gamma} (-1)^{n - c(\gamma)} \, w(\gamma),
\end{equation}
where the sum runs over all linear subdigraphs $\gamma$ of $\mathcal{D}(A)$ that involve all $n$ vertices.

\medskip

The \emph{length} of a linear subdigraph $\gamma$, denoted $L(\gamma)$, is the total number of edges in $\gamma$.  
Let $L_r$ denote the set of all linear subdigraphs of $\mathcal{D}(A)$ having length $r$.  
We then define
\begin{equation}\label{Eqn:Defn of lr;sgn sum of weightd of lsd of length r}
\ell_r \triangleq \sum_{\gamma \in L_r} (-1)^{c(\gamma)} \, w(\gamma).
\end{equation}

\smallskip

Alternatively, the characteristic polynomial of $A$ can also be expressed in terms of the sums of its principal minors as
\begin{equation}\label{eqn:C-H-Thm-f-as coef}
p_A(\lambda) = \lambda^n - f_1 \lambda^{n-1} + f_2 \lambda^{n-2} - \cdots + (-1)^{n-k} f_{n-k} \lambda^{k} + \cdots + (-1)^n f_n,
\end{equation}
where $f_i$ denotes the sum of all principal minors of order $i$ of $A$.  
Comparing Equations~\eqref{Eqn:c-h-thm,coeef-as-d} and~\eqref{eqn:C-H-Thm-f-as coef}, we obtain
\[
d_i = (-1)^i f_i, \qquad \text{for each } i \in [n].
\]
Since $f_i$ is the sum of all principal minors of order $i$, it can equivalently be expressed as a weighted sum over linear subdigraphs of $\mathcal{D}(A)$:
\begin{equation}\label{Eqn:fi=sum lsd}
f_i = \sum_{\gamma} (-1)^{i - c(\gamma)} \, w(\gamma),
\end{equation}
where the summation extends over all linear subdigraphs $\gamma$ of $\mathcal{D}(A)$ containing $i$ vertices.

\begin{lemma}\label{Lemma:Coeeffiient of Cha poly of A is sign sum of LSD}
Let $p_A(\lambda)$ be the characteristic polynomial of $A$ as in~\eqref{Eqn:c-h-thm,coeef-as-d}, and let $\ell_i$ be defined by~\eqref{Eqn:Defn of lr;sgn sum of weightd of lsd of length r}.  
Then, for each $i \in [n]$, we have
\[
d_i = \ell_i.
\]
\end{lemma}

\begin{proof}
We compute:
\begin{align*}
d_i 
  &= (-1)^i f_i \\
  &= (-1)^i \sum_{\gamma} (-1)^{i - c(\gamma)} w(\gamma) 
     && \text{(by Equation~\eqref{Eqn:fi=sum lsd})}\\
  &= \sum_{\gamma} (-1)^{c(\gamma)} w(\gamma)\\
  &= \ell_i, 
     && \text{(by Equation~\eqref{Eqn:Defn of lr;sgn sum of weightd of lsd of length r}).}
\end{align*}
\end{proof}

\begin{remark}\label{Rem:dr=lr: Coeffiecients of Ch eqn is weightd sum of lsd}
The preceding discussion provides a clear combinatorial interpretation of the coefficients of the characteristic polynomial of $A$:  
each coefficient is a sign-weighted sum of the weights of all linear subdigraphs of the associated digraph $\mathcal{D}(A)$.
\end{remark}

A \emph{walk} $c$ in the directed graph $\mathcal{D}(A)$ from a vertex $u$ to a vertex $v$ is defined as a sequence of vertices $u=x_0, x_1, \cdots, x_{k-1}, x_k=v$, where each pair $(x_i, x_{i+1})$ represents an edge for $i=0, 1, 2, \cdots, k-1$. This walk is termed closed if $u$ is equal to $v$. The length $L(c)$ of the walk $c$ corresponds to the total number of edges included in it. The weight $w(c)$ of the walk is calculated as the product of the weights of all edges that are part of the walk. It is important to note that in the case of a closed walk, the starting and ending points are inherently defined. We will denote the total weight of all closed walks of length $r$ as $c_r$. For example, some closed walks derived from the graph $D(A)$ in Figure \ref{fig:D(A)} are described in Figures \ref{fig:walk-1},  \ref{fig:walks-2-self-loop}, \ref{fig:walks-2-2cycle}, and \ref{fig:walks-3}.
 
 \begin{lemma}\label{Lemma:(i,j)th entry of Ak is the sum of weights of all closed walk lengh k from i to j}[Theorem 3.1.2, \cite{Brualdi-Cvetkovic-combinatorial-matrix-thy}]
Let $A=[a_{ij}]$ be a matrix of order $n.$ For each positive integer $k,$ the entry $a_{ij}^{(k)}$ of $A^k$ in the $i$th row and $j$th column equals the sum of the weights of all walks in $\mathcal{D}(A)$ of length $k$ from vertex $i$ to vertex $j.$    
 \end{lemma}
\begin{lemma}\label{Lemma:Tr(Ak)=Ck}
For any $k\in \mathbb{N},$ the trace of the matrix $A^k$ is $\text{Tr}(A^k)=c_k.$    
\end{lemma}
\begin{proof}
By Lemma \ref{Lemma:(i,j)th entry of Ak is the sum of weights of all closed walk lengh k from i to j}, the entry $a_{ii}^{(k)}$ of $A^k$ in the $i$th row and $i$th column equals the sum of the weights of all closed walks in $\mathcal{D}(A)$ of length $k$ from the vertex $i$ to itself. Therefore, $\text{Tr}(A^k)$ is the sum of the weights of all closed walks of length $k.$ That is, $\text{Tr}(A^k)=c_k.$          
\end{proof}
\begin{remark}\label{Rem:Combinatorial inte of Trace}
According to Lemma \ref{Lemma:Tr(Ak)=Ck}, a combinatorial interpretation of trace of a matrix is the sum of the weights of closed walks. 
\end{remark}
\section{Trace Cayley-Hamilton theorem}
In this section, we will articulate and demonstrate the trace Cayley-Hamilton theorem. The theorem, as referenced in \cite{Darij-Grinberg-Trace-Cayley-Hamilton}
\begin{theorem}[Trace Cayley-Hamilton theorem]\label{Thm: Trace cayley Hamilton Thm}
Let $A$ be a $n\times n$ matrix over the commutative ring $\mathbb{K}.$ Let $p_A(\lambda)$ be the characteristic polynomial of $A$ defined as in Equation \eqref{Eqn:c-h-thm,coeef-as-d}. Then the trace Cayley-Hamilton theorem says the following:
\begin{enumerate}
    \item $\text{Tr}(A^r)+\text{Tr}(A^{r-1})d_1+\cdots+\text{Tr}(A^{r-(n-1)})d_{n-1}+\text{Tr}(A^{r-n})d_n=0, \text{ for }r>n$
    \item $\text{Tr}(A^r)+\text{Tr}(A^{r-1})d_1+\cdots+\text{Tr}(A^{r-(n-1)})d_{n-1}+rd_r=0, \text{ for } 1\leq r\leq n.$
\end{enumerate}
\end{theorem}
To establish Theorem \ref{Thm: Trace cayley Hamilton Thm}, we will initially prove the following theorem. 
\begin{theorem}\label{Thm:Combinatorial proof of Trace Cayley Hamilton thm}
Let $A$ be an $n\times n$ matrix and $\mathcal{D}(A)$ be the associate weighted digraph of $A.$ Then the following hold: 
\begin{enumerate}
	\item $c_r+c_{r-1}\ell_1+c_{r-2}\ell_2+\cdots+c_{r-n}\ell_n=0, r>n$
	\item $c_r+c_{r-1}\ell_1+c_{r-2}\ell_2+\cdots+r\ell_r=0, 1\leq r\leq n$.
\end{enumerate}    
\end{theorem}

\begin{example}[Complete Analysis for $n=2$]

Consider the $2 \times 2$ matrix:
\[
A = \begin{pmatrix}
a & b \\
c & d
\end{pmatrix}
\]
The corresponding weighted digraph $\mathcal{D}(A)$ is represented by Figure \ref{fig:D(A)}.
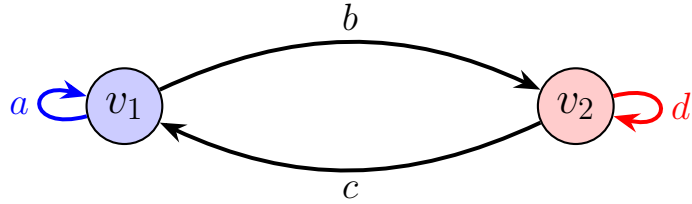
\begin{figure}[H]
    \centering

\begin{tikzpicture}[>=Stealth, thick, scale=2]
    \node[circle, draw=black, fill=blue!20, minimum size=1cm, font=\Large] (v1) at (0, 0) {$v_1$};
    \node[circle, draw=black, fill=red!20, minimum size=1cm, font=\Large] (v2) at (3, 0) {$v_2$};
    
    \draw[->, line width=1.5pt, blue] (v1) to[loop left, looseness=8] node[left, font=\large] {$a$} (v1);
    \draw[->, line width=1.5pt, red] (v2) to[loop right, looseness=8] node[right, font=\large] {$d$} (v2);
    
    \draw[->, line width=1.5pt, black] (v1) to[bend left=25] node[above, font=\large] {$b$} (v2);
    \draw[->, line width=1.5pt, black] (v2) to[bend left=25] node[below, font=\large] {$c$} (v1);
\end{tikzpicture}
\caption{The weighted directed graph $D(A)$ associated with the matrix $A.$}
    \label{fig:D(A)}
\end{figure}
A linear subdigraph with 1 vertex consists of a single cycle (self-loop). Figure \ref{fig:lsd-1} shows all such subdigraphs.

\begin{figure}[H]
    \centering
\begin{tikzpicture}[>=Stealth, thick, scale=1.5]
    \begin{scope}[xshift=0cm]
        \node[circle, draw=black, fill=blue!20, minimum size=0.8cm] (v1) at (0, 0) {$v_1$};
         \draw[->, line width=1.5pt, blue] (v1) to[loop left, looseness=6] node[left, font=\large] {$a$} (v1);
        \node[below] at (0, -0.8) {Weight: $a$};
        \node[below] at (0, -1.2) {Sign: $(-1)^1 = -1$};
    \end{scope}
    
    \begin{scope}[xshift=3.5cm]
        \node[circle, draw=black, fill=red!20, minimum size=0.8cm] (v2) at (0, 0) {$v_2$};
        \draw[->, line width=1.5pt, red] (v2) to[loop right, looseness=6] node[right, font=\large] {$d$} (v2);
        \node[below] at (0, -0.8) {Weight: $d$};
        \node[below] at (0, -1.2) {Sign: $(-1)^1 = -1$};
    \end{scope}
\end{tikzpicture}
\caption{Linear subdigraphs of length $1.$}
    \label{fig:lsd-1}
\end{figure}

Therefore, 
$\ell_1 = (-1)^1(a + d) = -(a + d).$ 

A linear subdigraph with 2 vertices can be either two disjoint 1-cycles or one 2-cycle. Figure \ref{fig:lsd-2} shows both types.
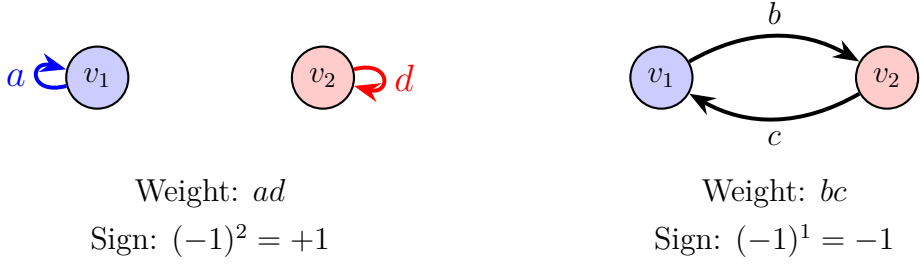
\begin{figure}[H]
    \centering
    \begin{tikzpicture}[>=Stealth, thick, scale=1.5]
    \begin{scope}[xshift=0cm]
        \node at (1, 1.5) {\textbf{Type 1:} Two self-loops};
        \node[circle, draw=black, fill=blue!20, minimum size=0.7cm] (v1) at (0, 0) {$v_1$};
        \node[circle, draw=black, fill=red!20, minimum size=0.7cm] (v2) at (2, 0) {$v_2$};
        \draw[->, line width=1.5pt, blue] (v1) to[loop left, looseness=6] node[left, font=\large] {$a$} (v1);
    
        \draw[->, line width=1.5pt, red] (v2) to[loop right, looseness=6] node[right, font=\large] {$d$} (v2);
    
        \node[below] at (1, -0.8) {Weight: $ad$};
        \node[below] at (1, -1.2) {Sign: $(-1)^2 = +1$};
    \end{scope}
    
    \begin{scope}[xshift=5cm]
        \node at (1, 1.5) {\textbf{Type 2:} One 2-cycle};
        \node[circle, draw=black, fill=blue!20, minimum size=0.7cm] (v1) at (0, 0) {$v_1$};
        \node[circle, draw=black, fill=red!20, minimum size=0.7cm] (v2) at (2, 0) {$v_2$};
        \draw[->, line width=1.5pt, black] (v1) to[bend left=30] node[above] {$b$} (v2);
        \draw[->, line width=1.5pt, black] (v2) to[bend left=30] node[below] {$c$} (v1);
        \node[below] at (1, -0.8) {Weight: $bc$};
        \node[below] at (1, -1.2) {Sign: $(-1)^1 = -1$};
    \end{scope}
\end{tikzpicture}
\caption{Linear subdigraphs of length $2.$}
    \label{fig:lsd-2}
\end{figure}
Note that, $\ell_2 = (+1)(ad) + (-1)(bc) = ad - bc.$ 

A closed walk of length 1 is simply a self-loop. Figure \ref{fig:walk-1} shows all such walks.

\begin{figure}[H]
    \centering

\begin{tikzpicture}[>=Stealth, thick, scale=1.5]
    \begin{scope}[xshift=0cm]
        \node[circle, draw=black, fill=blue!20, minimum size=0.8cm] (v1) at (0, 0) {$v_1$};
         \draw[->, line width=1.5pt, blue] (v1) to[loop left, looseness=6] node[left, font=\large] {$a$} (v1);
    \node[below] at (0, -0.7) {Weight: $a$};
    \end{scope}
    
    \begin{scope}[xshift=2.5cm]
        \node[circle, draw=black, fill=red!20, minimum size=0.8cm] (v2) at (0, 0) {$v_2$};
                \draw[->, line width=1.5pt, red] (v2) to[loop right, looseness=6] node[right, font=\large] {$d$} (v2);
    
        \node[below] at (0, -0.7) {Weight: $d$};
    \end{scope}
\end{tikzpicture}
\caption{Closed walks of length $1.$}
    \label{fig:walk-1}
\end{figure}
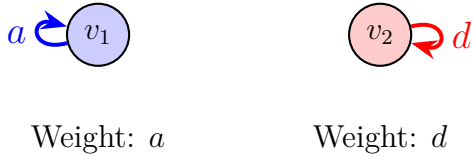
So, $c_1 = a + d = \text{tr}(A).$ 

A closed walk of length 2 can be formed by:
\begin{itemize}
    \item Using a self-loop twice (Figure \ref{fig:walks-2-self-loop})
    \item Going to another vertex and returning (Figure \ref{fig:walks-2-2cycle}).
\end{itemize}

\begin{figure}[H]
    \centering

\begin{tikzpicture}[>=Stealth, thick, scale=1.3]
    \begin{scope}[xshift=0cm]
    \node at (-.5, 1.2) {\small Path: $v_1 \xrightarrow{a} v_1 \xrightarrow{a} v_1$};
    \node at (3, 1.2) {\small Path: $v_2 \xrightarrow{d} v_2 \xrightarrow{d} v_2$};
        \node[circle, draw=black, fill=blue!20, minimum size=0.7cm] (v1) at (0, 0) {$v_1$};
         \draw[->, line width=1.5pt, blue] (v1) to[loop left, looseness=6] node[left, font=\large] {$a$} (v1);
    
        \node[below] at (0, -0.7) {Weight: $a^2$};
    \end{scope}
    
    \begin{scope}[xshift=2.5cm]
        \node[circle, draw=black, fill=red!20, minimum size=0.7cm] (v2) at (0, 0) {$v_2$};
        \draw[->, line width=1.5pt, red] (v2) to[loop right, looseness=6] node[right, font=\large] {$d$} (v2);
    
        \node[below] at (0, -0.7) {Weight: $d^2$};
    \end{scope}
\end{tikzpicture}
\caption{Closed walks of length $1$ formed by self-loop twice.}
    \label{fig:walks-2-self-loop}
\end{figure}
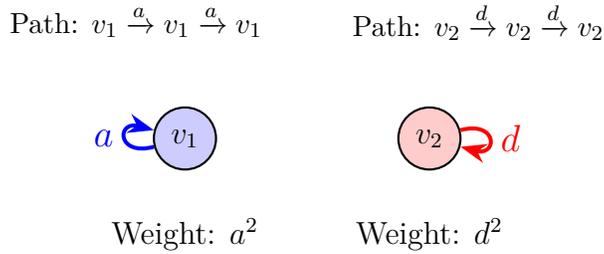

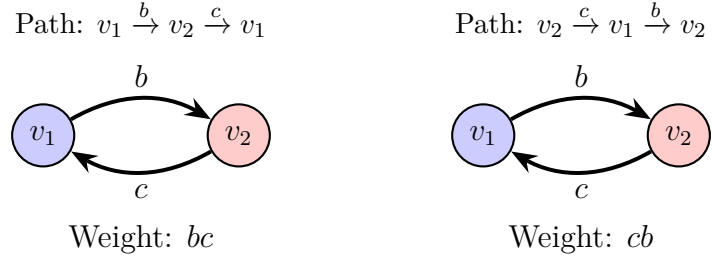
\begin{figure}[H]
    \centering

\begin{tikzpicture}[>=Stealth, thick, scale=1.3]
    \begin{scope}[xshift=0cm]
        \node at (1, 1.2) {\small Path: $v_1 \xrightarrow{b} v_2 \xrightarrow{c} v_1$};
        \node[circle, draw=black, fill=blue!20, minimum size=0.7cm] (v1) at (0, 0) {$v_1$};
        \node[circle, draw=black, fill=red!20, minimum size=0.7cm] (v2) at (2, 0) {$v_2$};
        \draw[->, line width=1.5pt] (v1) to[bend left=30] node[above] {$b$} (v2);
        \draw[->, line width=1.5pt] (v2) to[bend left=30] node[below] {$c$} (v1);
        \node[below] at (1, -0.8) {Weight: $bc$};
    \end{scope}
    
    \begin{scope}[xshift=4.5cm]
        \node at (1, 1.2) {\small Path: $v_2 \xrightarrow{c} v_1 \xrightarrow{b} v_2$};
        \node[circle, draw=black, fill=blue!20, minimum size=0.7cm] (v1) at (0, 0) {$v_1$};
        \node[circle, draw=black, fill=red!20, minimum size=0.7cm] (v2) at (2, 0) {$v_2$};
        \draw[->, line width=1.5pt] (v2) to[bend left=30] node[below] {$c$} (v1);
        \draw[->, line width=1.5pt] (v1) to[bend left=30] node[above] {$b$} (v2);
        \node[below] at (1, -0.8) {Weight: $cb$};
    \end{scope}
\end{tikzpicture}
\caption{Closed walks of length $2$ formed by two length cycle.}
    \label{fig:walks-2-2cycle}
\end{figure}
Therefore, $c_2 = a^2 + d^2 + bc + cb = a^2 + d^2 + 2bc = \text{tr}(A^2).$

Note that for $r = 1,$ 
\begin{align*}
c_1 + \ell_1 &= (a + d) + (-(a + d)) \\
&= a + d - a - d \\
&= 0.
\end{align*}
Now, for $r = 2,$ we will show the the identity $c_2 + c_1\ell_1 + 2\ell_2 = 0.$
\begin{align*}
c_2 + c_1\ell_1 + 2\ell_2 &= (a^2 + d^2 + 2bc) + (a+d)(-(a+d)) + 2(ad - bc) \\
&= a^2 + d^2 + 2bc - (a^2 + 2ad + d^2) + 2ad - 2bc \\
&= a^2 + d^2 + 2bc - a^2 - 2ad - d^2 + 2ad - 2bc \\
&= 0 .
\end{align*}

Now, we consider that $r=3>n=2.$ Here we will check that $c_3 + c_2\ell_1 + c_1\ell_2 = 0.$ For that first, we have to compute $c_3.$ There are 8 distinct closed walks of length 3, shown in Figure \ref{fig:walks-3}:
\begin{itemize}
    \item 2 walks using only self-loops (Type 1)
    \item 3 walks using self-loop at $v_1$ plus the 2-cycle (Type 2)
    \item 3 walks using self-loop at $v_2$ plus the 2-cycle (Type 3).
\end{itemize}
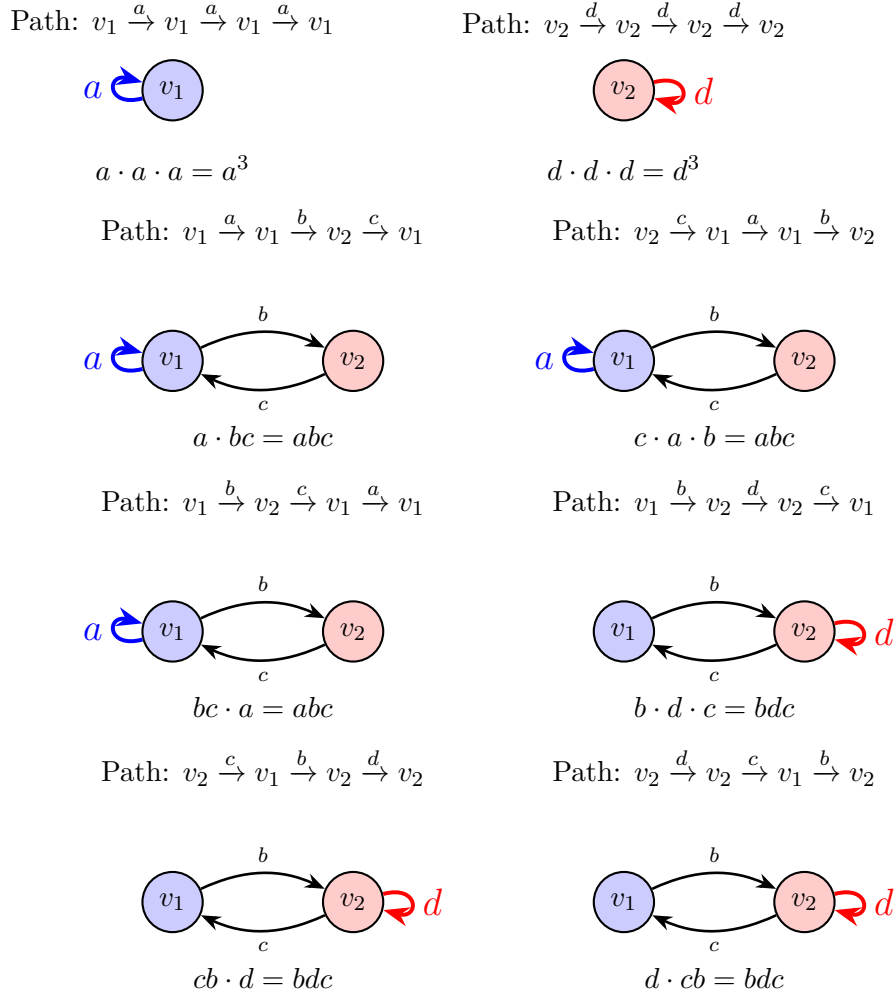
\begin{figure}
    \centering
   
\begin{tikzpicture}[>=Stealth, thick, scale=1.2]
    \begin{scope}[xshift=0cm]
        \node at (0, .8) {\small Path: $v_1 \xrightarrow{a} v_1 \xrightarrow{a} v_1\xrightarrow{a} v_1$};
        \node[circle, draw=black, fill=blue!20, minimum size=0.6cm, font=\small] (v1) at (0, 0) {$v_1$};
        \draw[->, line width=1.5pt, blue] (v1) to[loop left, looseness=6] node[left, font=\large] {$a$} (v1);
    
        \node[below] at (0, -0.6) {\small $a\cdot a\cdot a=a^3$};
    \end{scope}
    
    \begin{scope}[xshift=5cm]
        \node at (0, .8) {\small Path: $v_2 \xrightarrow{d} v_2 \xrightarrow{d} v_2\xrightarrow{d} v_2$};
        \node[circle, draw=black, fill=red!20, minimum size=0.6cm, font=\small] (v2) at (0, 0) {$v_2$};
       \draw[->, line width=1.5pt, red] (v2) to[loop right, looseness=6] node[right, font=\large] {$d$} (v2);
    
        \node[below] at (0, -0.6) {\small $d\cdot d\cdot d=d^3$};
    \end{scope}

    \begin{scope}[xshift=0cm, yshift=-3cm]
        \node at (1, 1.5) {\small Path: $v_1 \xrightarrow{a} v_1 \xrightarrow{b} v_2 \xrightarrow{c} v_1$};
        \node[circle, draw=black, fill=blue!20, minimum size=0.6cm, font=\small] (v1) at (0, 0) {$v_1$};
        \node[circle, draw=black, fill=red!20, minimum size=0.6cm, font=\small] (v2) at (2, 0) {$v_2$};
       \draw[->, line width=1.5pt, blue] (v1) to[loop left, looseness=6] node[left, font=\large] {$a$} (v1);
    
        \draw[->, line width=1pt] (v1) to[bend left=25] node[above, font=\tiny] {$b$} (v2);
        \draw[->, line width=1pt] (v2) to[bend left=25] node[below, font=\tiny] {$c$} (v1);
        \node[below] at (1, -0.6) {\small $a \cdot bc=abc$};
    \end{scope}
    
    \begin{scope}[xshift=5cm, yshift=-6cm]
        \node at (1, 1.5) {\small Path: $v_1 \xrightarrow{b} v_2 \xrightarrow{d} v_2 \xrightarrow{c} v_1$};
        \node[circle, draw=black, fill=blue!20, minimum size=0.6cm, font=\small] (v1) at (0, 0) {$v_1$};
        \node[circle, draw=black, fill=red!20, minimum size=0.6cm, font=\small] (v2) at (2, 0) {$v_2$};
        \draw[->, line width=1.5pt, red] (v2) to[loop right, looseness=6] node[right, font=\large] {$d$} (v2);
    
        \draw[->, line width=1pt] (v1) to[bend left=25] node[above, font=\tiny] {$b$} (v2);
        \draw[->, line width=1pt] (v2) to[bend left=25] node[below, font=\tiny] {$c$} (v1);
        \node[below] at (1, -0.6) {\small $b \cdot d\cdot c=bdc$};
    \end{scope}

    \begin{scope}[yshift=-6cm]
        \node at (1, 1.5) {\small Path: $v_1 \xrightarrow{b} v_2 \xrightarrow{c} v_1 \xrightarrow{a} v_1$};
        \node[circle, draw=black, fill=blue!20, minimum size=0.6cm, font=\small] (v1) at (0, 0) {$v_1$};
        \node[circle, draw=black, fill=red!20, minimum size=0.6cm, font=\small] (v2) at (2, 0) {$v_2$};
        \draw[->, line width=1.5pt, blue] (v1) to[loop left, looseness=6] node[left, font=\large] {$a$} (v1);
    
        \draw[->, line width=1pt] (v1) to[bend left=25] node[above, font=\tiny] {$b$} (v2);
        \draw[->, line width=1pt] (v2) to[bend left=25] node[below, font=\tiny] {$c$} (v1);
        \node[below] at (1, -0.6) {\small $bc \cdot a=abc$};
    \end{scope}
    
 \begin{scope}[xshift=5cm, yshift=-3cm]
        \node at (1, 1.5) {\small Path: $v_2 \xrightarrow{c} v_1 \xrightarrow{a} v_1 \xrightarrow{b} v_2$};
        \node[circle, draw=black, fill=blue!20, minimum size=0.6cm, font=\small] (v1) at (0, 0) {$v_1$};
        \node[circle, draw=black, fill=red!20, minimum size=0.6cm, font=\small] (v2) at (2, 0) {$v_2$};
    
       \draw[->, line width=1.5pt, blue] (v1) to[loop left, looseness=6] node[left, font=\large] {$a$} (v1);
    
        \draw[->, line width=1pt] (v1) to[bend left=25] node[above, font=\tiny] {$b$} (v2);
        \draw[->, line width=1pt] (v2) to[bend left=25] node[below, font=\tiny] {$c$} (v1);
        \node[below] at (1, -0.6) {\small $c\cdot a \cdot b=abc$};
    \end{scope}   

    \begin{scope}[xshift=0cm, yshift=-9cm]
        \node at (1, 1.5) {\small Path: $v_2 \xrightarrow{c} v_1 \xrightarrow{b} v_2 \xrightarrow{d} v_2$};
        \node[circle, draw=black, fill=blue!20, minimum size=0.6cm, font=\small] (v1) at (0, 0) {$v_1$};
        \node[circle, draw=black, fill=red!20, minimum size=0.6cm, font=\small] (v2) at (2, 0) {$v_2$};
        \draw[->, line width=1.5pt, red] (v2) to[loop right, looseness=6] node[right, font=\large] {$d$} (v2);
    
        \draw[->, line width=1pt] (v1) to[bend left=25] node[above, font=\tiny] {$b$} (v2);
        \draw[->, line width=1pt] (v2) to[bend left=25] node[below, font=\tiny] {$c$} (v1);
        \node[below] at (1, -0.6) {\small $cb \cdot d=bdc$};
    \end{scope}  

    \begin{scope}[xshift=5cm, yshift=-9cm]
        \node at (1, 1.5) {\small Path: $v_2 \xrightarrow{d} v_2 \xrightarrow{c} v_1 \xrightarrow{b} v_2$};
        \node[circle, draw=black, fill=blue!20, minimum size=0.6cm, font=\small] (v1) at (0, 0) {$v_1$};
        \node[circle, draw=black, fill=red!20, minimum size=0.6cm, font=\small] (v2) at (2, 0) {$v_2$};
        \draw[->, line width=1.5pt, red] (v2) to[loop right, looseness=6] node[right, font=\large] {$d$} (v2);
        \draw[->, line width=1pt] (v1) to[bend left=25] node[above, font=\tiny] {$b$} (v2);
        \draw[->, line width=1pt] (v2) to[bend left=25] node[below, font=\tiny] {$c$} (v1);
        \node[below] at (1, -0.6) {\small $d \cdot cb=bdc$};
    \end{scope}  
\end{tikzpicture}
    \caption{All closed walks of length $3.$}
    \label{fig:walks-3}
\end{figure}
Therefore, 
$c_3 = a^3 + d^3 + 3abc + 3dbc = a^3 + d^3 + 3bc(a + d).$
Now verify the identity:
\begin{align*}
c_3 + c_2\ell_1 + c_1\ell_2 &= [a^3 + d^3 + 3bc(a+d)] + [a^2 + d^2 + 2bc][-(a+d)] \\
&\quad + [a+d][ad - bc] \\
&= a^3 + d^3 + 3abc + 3dbc \\
&\quad - a^3 - a^2d - ad^2 - d^3 - 2abc - 2dbc \\
&\quad + a^2d + ad^2 - abc - dbc \\
&= a^3 + d^3 + 3abc + 3dbc - a^3 - d^3 - 2abc - 2dbc - abc - dbc \\
&= 0. 
\end{align*}
\end{example}

\section{Proof of the Trace Cayley-Hamilton theorem}
This section is devoted to the proof of our main theorem. We begin by proving Theorem~\ref{Thm:Combinatorial proof of Trace Cayley Hamilton thm}, which serves as a key intermediate result. 
\begin{proof}[Proof of Theorem \ref{Thm:Combinatorial proof of Trace Cayley Hamilton thm}]
Initially, we address the scenario where $r>n$. To establish this, we examine all ordered pairs $(c, \gamma)$, where $c$ represents a closed walk and $\gamma$ is a linear subdigraph (which may be empty), such that $L(c)+L(\gamma)=r$. Define the weight $W$ of $(c, \gamma)$ to be $W((c, \gamma))=(-1)^{c(\gamma)}w(c)w(\gamma)$. It is important to observe that the left-hand side of $(1)$ is precisely equal to $\sum\limits_{(c, \gamma)}W((c, \gamma))$, where the summation runs over all ordered pairs $(c, \gamma)$ as previously described.

The key observation is that, given $r>n$, either $c$ and $\gamma$ share a common vertex, or $c$ does not represent a ``simple'' closed walk (where ``simple'' indicates that the structure of the closed walk is a directed cycle). Consider a specific pair $(c, \gamma)$ that meets these criteria. Let $x$ denote both the starting and ending vertex of $c$. As we traverse from $x$ along $c$, there are two scenarios to consider: either we first encounter a vertex $y$ that is part of $\gamma$, or we complete a closed directed cycle $\acute{c}$, which is a subwalk of $c$, without encountering any vertex of $\gamma$ during our journey from $x$ to the completion of $\acute{c}$.

\subsection*{Scenario 1: Encountering a vertex $y \in \gamma:$}

In the first scenario, we create a new ordered pair $(\tilde{c}, \tilde{\gamma})$, where $\tilde{c}=\widehat{xy}|_{c}\bigodot \gamma_{y}\bigodot \widehat{yx}|_{c}$ and $\tilde{\gamma}=\gamma\setminus\{\gamma_y\}$. Here, $\widehat{xy}|_{c}$ represents the walk from $x$ to $y$ along $c$, and $\gamma_y$ is the directed cycle of $\gamma$ that includes vertex $y$. It is important to note that $W((\tilde{c}, \tilde{\gamma}))=-W((c, \gamma))$.

\subsection*{Scenario 2: Completing a simple cycle $\acute{c}$ without encountering $\gamma:$}

In the second scenario, we form a new ordered pair $(\tilde{\tilde{c}}, \tilde{\tilde{\gamma}})$, where $\tilde{\tilde{c}}$ is created by removing the directed cycle $\acute{c}$ from $c$, and $\tilde{\tilde{\gamma}}$ is defined as $\gamma\cup \acute{c}$. This demonstrates that the process is indeed an involution and is sign-reversing, as previously observed. It is also noteworthy that $W((\tilde{\tilde{c}}, \tilde{\tilde{\gamma}}))=-W((c, \gamma))$. This concludes the proof of the case $r>n$.


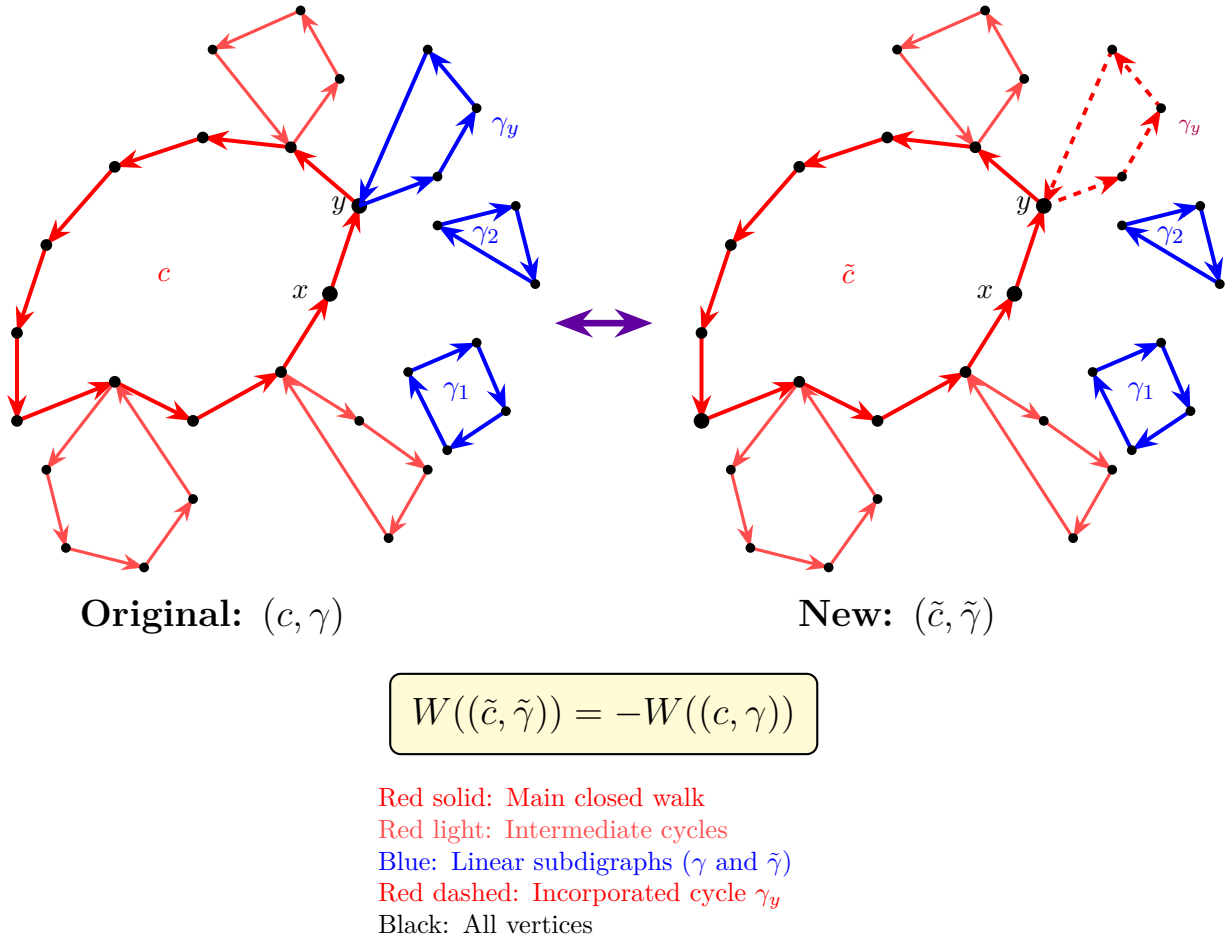
\begin{figure}[H]
\centering
\begin{tikzpicture}[>=Stealth, thick, scale=1.3]
    
    \begin{scope}[xshift=0cm]
        \node[font=\large\bfseries] at (2, -2) {Original: $(c, \gamma)$};
        
        \draw[red, line width=1.5pt, ->] (0,0) -- (1,0.4);
        \draw[red, line width=1.5pt, ->] (1,0.4) -- (1.8,0);
        \draw[red, line width=1.5pt, ->] (1.8,0) -- (2.7,0.5);
        \draw[red, line width=1.5pt, ->] (2.7,0.5) -- (3.2,1.3);
        \draw[red, line width=1.5pt, ->] (3.2,1.3) -- (3.5,2.2);
        \draw[red, line width=1.5pt, ->] (3.5,2.2) -- (2.8,2.8);
        \draw[red, line width=1.5pt, ->] (2.8,2.8) -- (1.9,2.9);
        \draw[red, line width=1.5pt, ->] (1.9,2.9) -- (1,2.6);
        \draw[red, line width=1.5pt, ->] (1,2.6) -- (0.3,1.8);
        \draw[red, line width=1.5pt, ->] (0.3,1.8) -- (0,0.9);
        \draw[red, line width=1.5pt, ->] (0,0.9) -- (0,0);
        
        \draw[red!70, line width=1.2pt, ->] (1,0.4) -- (0.3,-0.5);
        \draw[red!70, line width=1.2pt, ->] (0.3,-0.5) -- (0.5,-1.3);
        \draw[red!70, line width=1.2pt, ->] (0.5,-1.3) -- (1.3,-1.5);
        \draw[red!70, line width=1.2pt, ->] (1.3,-1.5) -- (1.8,-0.8);
        \draw[red!70, line width=1.2pt, ->] (1.8,-0.8) -- (1,0.4);
        
        \draw[red!70, line width=1.2pt, ->] (2.7,0.5) -- (3.5,0);
        \draw[red!70, line width=1.2pt, ->] (3.5,0) -- (4.2,-0.5);
        \draw[red!70, line width=1.2pt, ->] (4.2,-0.5) -- (3.8,-1.2);
        \draw[red!70, line width=1.2pt, ->] (3.8,-1.2) -- (2.7,0.5);
        
        \draw[red!70, line width=1.2pt, ->] (2.8,2.8) -- (3.3,3.5);
        \draw[red!70, line width=1.2pt, ->] (3.3,3.5) -- (2.9,4.2);
        \draw[red!70, line width=1.2pt, ->] (2.9,4.2) -- (2,3.8);
        \draw[red!70, line width=1.2pt, ->] (2,3.8) -- (2.8,2.8);
        
        \fill[black] (0,0) circle (0.06);
        \node[below left, black, font=\small] at (3.1,1.5) {$x$};
        
        \fill[black] (1,0.4) circle (0.06);
        \fill[black] (1.8,0) circle (0.06);
        \fill[black] (2.7,0.5) circle (0.06);
        \fill[black] (3.2,1.3) circle (0.08);
        
        \fill[black] (3.5,2.2) circle (0.08);
        \node[right, black, font=\small] at (3.1,2.2) {$y$};
        
        \fill[black] (2.8,2.8) circle (0.06);
        \fill[black] (1.9,2.9) circle (0.06);
        \fill[black] (1,2.6) circle (0.06);
        \fill[black] (0.3,1.8) circle (0.06);
        \fill[black] (0,0.9) circle (0.06);
        
        \fill[black] (0.3,-0.5) circle (0.05);
        \fill[black] (0.5,-1.3) circle (0.05);
        \fill[black] (1.3,-1.5) circle (0.05);
        \fill[black] (1.8,-0.8) circle (0.05);
        
        \fill[black] (3.5,0) circle (0.05);
        \fill[black] (4.2,-0.5) circle (0.05);
        \fill[black] (3.8,-1.2) circle (0.05);
        
        \fill[black] (3.3,3.5) circle (0.05);
        \fill[black] (2.9,4.2) circle (0.05);
        \fill[black] (2,3.8) circle (0.05);
        
        \draw[blue, line width=1.5pt, ->] (3.5,2.2) -- (4.3,2.5);
        \draw[blue, line width=1.5pt, ->] (4.3,2.5) -- (4.7,3.2);
        \draw[blue, line width=1.5pt, ->] (4.7,3.2) -- (4.2,3.8);
        \draw[blue, line width=1.5pt, ->] (4.2,3.8) -- (3.5,2.2);
        \node[blue, font=\small] at (5, 3) {$\gamma_y$};
        
        \fill[black] (4.3,2.5) circle (0.05);
        \fill[black] (4.7,3.2) circle (0.05);
        \fill[black] (4.2,3.8) circle (0.05);
        
        \draw[blue, line width=1.5pt, ->] (4, 0.5) -- (4.7, 0.8);
        \draw[blue, line width=1.5pt, ->] (4.7, 0.8) -- (5, 0.1);
        \draw[blue, line width=1.5pt, ->] (5, 0.1) -- (4.4, -0.3);
        \draw[blue, line width=1.5pt, ->] (4.4, -0.3) -- (4, 0.5);
        \node[blue, font=\small] at (4.5, 0.3) {$\gamma_1$};
        
        \fill[black] (4, 0.5) circle (0.05);
        \fill[black] (4.7, 0.8) circle (0.05);
        \fill[black] (5, 0.1) circle (0.05);
        \fill[black] (4.4, -0.3) circle (0.05);
        
        \draw[blue, line width=1.5pt, ->] (4.3, 2) -- (5.1, 2.2);
        \draw[blue, line width=1.5pt, ->] (5.1, 2.2) -- (5.3, 1.4);
        \draw[blue, line width=1.5pt, ->] (5.3, 1.4) -- (4.3, 2);
        \node[blue, font=\small] at (4.8, 1.9) {$\gamma_2$};
        
        \fill[black] (4.3, 2) circle (0.05);
        \fill[black] (5.1, 2.2) circle (0.05);
        \fill[black] (5.3, 1.4) circle (0.05);
        
        \node[red, font=\small] at (1.5, 1.5) {$c$};
    \end{scope}
    
    \draw[<->, line width=2.5pt, blue!50!purple] (5.5,1) -- (6.5, 1);
    
    \begin{scope}[xshift=7cm]
        \node[font=\large\bfseries] at (2, -2) {New: $(\tilde{c}, \tilde{\gamma})$};
        
        \draw[red, line width=1.5pt, ->] (0,0) -- (1,0.4);
        \draw[red, line width=1.5pt, ->] (1,0.4) -- (1.8,0);
        \draw[red, line width=1.5pt, ->] (1.8,0) -- (2.7,0.5);
        \draw[red, line width=1.5pt, ->] (2.7,0.5) -- (3.2,1.3);
        \draw[red, line width=1.5pt, ->] (3.2,1.3) -- (3.5,2.2);
        
        \draw[red, line width=1.5pt, ->, dashed] (3.5,2.2) -- (4.3,2.5);
        \draw[red, line width=1.5pt, ->, dashed] (4.3,2.5) -- (4.7,3.2);
        \draw[red, line width=1.5pt, ->, dashed] (4.7,3.2) -- (4.2,3.8);
        \draw[red, line width=1.5pt, ->, dashed] (4.2,3.8) -- (3.5,2.2);
        \node[red!70!blue, font=\tiny] at (5, 3) {$\gamma_y$};
        
        \draw[red, line width=1.5pt, ->] (3.5,2.2) -- (2.8,2.8);
        \draw[red, line width=1.5pt, ->] (2.8,2.8) -- (1.9,2.9);
        \draw[red, line width=1.5pt, ->] (1.9,2.9) -- (1,2.6);
        \draw[red, line width=1.5pt, ->] (1,2.6) -- (0.3,1.8);
        \draw[red, line width=1.5pt, ->] (0.3,1.8) -- (0,0.9);
        \draw[red, line width=1.5pt, ->] (0,0.9) -- (0,0);
        
        \draw[red!70, line width=1.2pt, ->] (1,0.4) -- (0.3,-0.5);
        \draw[red!70, line width=1.2pt, ->] (0.3,-0.5) -- (0.5,-1.3);
        \draw[red!70, line width=1.2pt, ->] (0.5,-1.3) -- (1.3,-1.5);
        \draw[red!70, line width=1.2pt, ->] (1.3,-1.5) -- (1.8,-0.8);
        \draw[red!70, line width=1.2pt, ->] (1.8,-0.8) -- (1,0.4);
        
        \draw[red!70, line width=1.2pt, ->] (2.7,0.5) -- (3.5,0);
        \draw[red!70, line width=1.2pt, ->] (3.5,0) -- (4.2,-0.5);
        \draw[red!70, line width=1.2pt, ->] (4.2,-0.5) -- (3.8,-1.2);
        \draw[red!70, line width=1.2pt, ->] (3.8,-1.2) -- (2.7,0.5);
        
        \draw[red!70, line width=1.2pt, ->] (2.8,2.8) -- (3.3,3.5);
        \draw[red!70, line width=1.2pt, ->] (3.3,3.5) -- (2.9,4.2);
        \draw[red!70, line width=1.2pt, ->] (2.9,4.2) -- (2,3.8);
        \draw[red!70, line width=1.2pt, ->] (2,3.8) -- (2.8,2.8);
        
        \fill[black] (0,0) circle (0.08);
        \node[below left, black, font=\small] at (3.1,1.5) {$x$};
        
        \fill[black] (1,0.4) circle (0.06);
        \fill[black] (1.8,0) circle (0.06);
        \fill[black] (2.7,0.5) circle (0.06);
        \fill[black] (3.2,1.3) circle (0.08);
        
        \fill[black] (3.5,2.2) circle (0.08);
        \node[right, black, font=\small] at (3.1,2.2) {$y$};
        
        \fill[black] (2.8,2.8) circle (0.06);
        \fill[black] (1.9,2.9) circle (0.06);
        \fill[black] (1,2.6) circle (0.06);
        \fill[black] (0.3,1.8) circle (0.06);
        \fill[black] (0,0.9) circle (0.06);
        
        \fill[black] (0.3,-0.5) circle (0.05);
        \fill[black] (0.5,-1.3) circle (0.05);
        \fill[black] (1.3,-1.5) circle (0.05);
        \fill[black] (1.8,-0.8) circle (0.05);
        
        \fill[black] (3.5,0) circle (0.05);
        \fill[black] (4.2,-0.5) circle (0.05);
        \fill[black] (3.8,-1.2) circle (0.05);
        
        \fill[black] (3.3,3.5) circle (0.05);
        \fill[black] (2.9,4.2) circle (0.05);
        \fill[black] (2,3.8) circle (0.05);
        
        \fill[black] (4.3,2.5) circle (0.05);
        \fill[black] (4.7,3.2) circle (0.05);
        \fill[black] (4.2,3.8) circle (0.05);
        
        \draw[blue, line width=1.5pt, ->] (4, 0.5) -- (4.7, 0.8);
        \draw[blue, line width=1.5pt, ->] (4.7, 0.8) -- (5, 0.1);
        \draw[blue, line width=1.5pt, ->] (5, 0.1) -- (4.4, -0.3);
        \draw[blue, line width=1.5pt, ->] (4.4, -0.3) -- (4, 0.5);
        \node[blue, font=\small] at (4.5, 0.3) {$\gamma_1$};
        
        \fill[black] (4, 0.5) circle (0.05);
        \fill[black] (4.7, 0.8) circle (0.05);
        \fill[black] (5, 0.1) circle (0.05);
        \fill[black] (4.4, -0.3) circle (0.05);
        
        \draw[blue, line width=1.5pt, ->] (4.3, 2) -- (5.1, 2.2);
        \draw[blue, line width=1.5pt, ->] (5.1, 2.2) -- (5.3, 1.4);
        \draw[blue, line width=1.5pt, ->] (5.3, 1.4) -- (4.3, 2);
        \node[blue, font=\small] at (4.8, 1.9) {$\gamma_2$};
        
        \fill[black] (4.3, 2) circle (0.05);
        \fill[black] (5.1, 2.2) circle (0.05);
        \fill[black] (5.3, 1.4) circle (0.05);
        
        \node[red, font=\small] at (1.5, 1.5) {$\tilde{c}$};
    \end{scope}
    
    \node[font=\large, draw, rounded corners, fill=yellow!20, inner sep=8pt] 
        at (6, -3) {$W((\tilde{c}, \tilde{\gamma})) = -W((c, \gamma))$};
        
    \node[text width=6cm, font=\footnotesize, align=left] at (6, -4.5) {
        \textcolor{red}{Red solid: Main closed walk}\\
        \textcolor{red!70}{Red light: Intermediate cycles}\\
        \textcolor{blue}{Blue: Linear subdigraphs ($\gamma$ and $\tilde{\gamma}$)}\\
        \textcolor{red}{Red dashed: Incorporated cycle $\gamma_y$}\\
        Black: All vertices
    };
\end{tikzpicture}
\caption{The closed walk $c$ with intermediate cycles encounters a vertex $y$ that belongs to cycle $\gamma_y$ in the linear subdigraph $\gamma$. The involution transforms the pair by incorporating the cycle $\gamma_y$ into the closed walk.}
\label{fig:closed_walk_involution}
\end{figure}
Now, we prove for the Case $r \leq n.$ Let $A = \{(c, \gamma) : c \text{ is a closed walk of length } \geq 1, \gamma \text{ is a linear subdigraph, and } L(c) + L(\gamma) = r\}$. Define the sum:
\[
S = \sum_{(c, \gamma) \in A} W((c, \gamma)) + r\ell_r,
\]
where $W((c, \gamma)) = (-1)^{c(\gamma)}w(c)w(\gamma)$. The left-hand side of equation (2) equals $S$. We partition $A$ into two disjoint subsets in the following way:
\begin{itemize}
    \item \text{BAD pairs} $B = \{(c, \gamma) \in A : c \cap \gamma \neq \varnothing \text{ or } c \text{ is not a simple cycle}\}$
    \item \text{GOOD pairs} $A \setminus B = \{(c, \gamma) \in A : c \cap \gamma = \varnothing \text{ and } c \text{ is a simple cycle}\}.$
\end{itemize}
Notice that the weights in $B$ cancel pairwise via the sign-reversing involution (same as in the $r > n$ case). Also, from the construction of the set $A\setminus B,$ it is easy to see that each GOOD pair $(c, \gamma) \in A \setminus B$ corresponds to a decomposition of a linear subdigraph $\dot{\gamma} = c \cup \gamma$ on the vertex set $\{v_1, v_2, \ldots, v_r\}$. Now, our claim is that, for each linear subdigraph $\dot{\gamma}$ on $r$ vertices, there are exactly $r$ GOOD pairs in $A \setminus B$.

In fact, for each vertex $v_i \in \{v_1, \ldots, v_r\},$ let $c_{v_i}$ be the cycle in $\dot{\gamma}$ containing vertex $v_i$. Now, set $\gamma_i = \dot{\gamma} \setminus \{c_{v_i}\}$ (remove this cycle), and form the pair $(c_{v_i}, \gamma_i).$ Clearly, $(c_{v_i}, \gamma_i) \in A \setminus B.$ Since we can choose any of the $r$ vertices as our starting point, we get exactly $r$ distinct GOOD pairs (see Figure \ref{fig:lsd-with-GOOD-pairs}, here we discuss it with one example).

\begin{figure}[H]
\centering
\begin{tikzpicture}[>=Stealth, thick, scale=1.3]
    
    \begin{scope}[xshift=-6cm, yshift=-2cm]
        \node at (1, -.7) {$\dot{\gamma}$};
        \draw[->, line width=2.5pt, blue!50!purple] (2.5,1.5) -- (3.9, 1.5);
       \node[circle, draw, fill=blue!20, minimum size=0.05] (v1) at (0, 2.2) {$v_1$};
        \node[circle, draw, fill=blue!20, minimum size=0.5cm] (v2) at (2, 2.8) {$v_2$};
        \node[circle, draw, fill=green!20, minimum size=0.5cm] (v3) at (0, 0.3) {$v_3$};
        \node[circle, draw, fill=green!20, minimum size=0.5cm] (v4) at (2, -0.2) {$v_4$};
        \node[circle, draw, fill=orange!20, minimum size=0.5cm] (v5) at (1.5, 1.5) {$v_5$};
        
        \draw[->, very thick, blue] (v1) -- (v2);
        \draw[->, very thick, blue] (v2) to[bend left=30] (v1);
        
        \draw[->, very thick, green!70!black] (v3) -- (v4);
        \draw[->, very thick, green!70!black] (v4) to[bend left=30] (v3);
        
        \draw[->, very thick, orange] (v5) to[loop right] (v5);
        
    \end{scope}
    
    
    \begin{scope}[xshift=-1cm]
        
        \begin{scope}[yshift=2cm]
            \node[anchor=west] at (.4,-.3) {$c_{v_1}$};
            \node[circle, draw, fill=blue!20, minimum size=0.05] (p1v1) at (0, 0) {$v_1$};
            \node[circle, draw, fill=blue!20, minimum size=0.05] (p1v2) at (1.2, 0.3) {$v_2$};
            \draw[->, very thick, blue] (p1v1) -- (p1v2);
            \draw[->, very thick, blue] (p1v2) to[bend left=25] (p1v1);
            
            \node[circle, draw, fill=green!20, minimum size=0.5cm] (p1v3) at (2.8, 0.2) {\tiny $v_3$};
            \node[circle, draw, fill=green!20, minimum size=0.5cm] (p1v4) at (3.8, -0.1) {\tiny $v_4$};
            \node[circle, draw, fill=orange!20, minimum size=0.5cm] (p1v5) at (4.7, 0.2) {\tiny $v_5$};
            \draw[->, thick, blue] (p1v3) -- (p1v4);
            \draw[->, thick, blue] (p1v4) to[bend left=20] (p1v3);
            \draw[->, very thick, orange] (p1v5) to[loop right, looseness=8] (p1v5);
            \node[blue, right] at (4.2, -.4) {$\gamma_1$};
        \end{scope}
        
        \begin{scope}[yshift=0.8cm]
            \node[anchor=west] at (.4,-.25) {$c_{v_2}$};
            \node[circle, draw, fill=blue!20, minimum size=0.05] (p2v1) at (0, 0) {$v_1$};
            \node[circle, draw, fill=blue!20, minimum size=0.05] (p2v2) at (1.2, 0.3) {$v_2$};
            \draw[->, very thick, blue] (p2v2) -- (p2v1);
            \draw[->, very thick, blue] (p2v1) to[bend left=25] (p2v2);
            
            \node[circle, draw, fill=green!20, minimum size=0.5cm] (p2v3) at (2.8, 0.2) {\tiny $v_3$};
            \node[circle, draw, fill=green!20, minimum size=0.5cm] (p2v4) at (3.8, -0.1) {\tiny $v_4$};
            \node[circle, draw, fill=orange!20, minimum size=0.5cm] (p2v5) at (4.7, 0.2) {\tiny $v_5$};
            \draw[->, thick, blue] (p2v3) -- (p2v4);
            \draw[->, thick, blue] (p2v4) to[bend left=20] (p2v3);
            \draw[->, very thick, orange] (p2v5) to[loop right, looseness=8] (p2v5);
            \node[blue, right] at (4.2, -.4) {$\gamma_2$};
        \end{scope}
        
        \begin{scope}[yshift=-0.4cm]
            \node[anchor=west] at (.4,-.3) {$c_{v_3}$};
            \node[circle, draw, fill=green!20, minimum size=0.5cm] (p3v3) at (0, 0) {$v_3$};
            \node[circle, draw, fill=green!20, minimum size=0.5cm] (p3v4) at (1.2, 0.3) {$v_4$};
            \draw[->, very thick, green!70!black] (p3v3) -- (p3v4);
            \draw[->, very thick, green!70!black] (p3v4) to[bend left=25] (p3v3);
            
            \node[circle, draw, fill=blue!20, minimum size=0.05] (p3v1) at (2.8, 0.3) {\tiny $v_1$};
            \node[circle, draw, fill=blue!20, minimum size=0.05] (p3v2) at (3.8, -0.1) {\tiny $v_2$};
            \node[circle, draw, fill=orange!20, minimum size=0.5cm] (p3v5) at (4.7, 0.2) {\tiny $v_5$};
            \draw[->, thick, blue] (p3v1) -- (p3v2);
            \draw[->, thick, blue] (p3v2) to[bend left=20] (p3v1);
            \draw[->, very thick, orange] (p3v5) to[loop right, looseness=8] (p3v5);
            \node[blue, right] at (4.2, -.4) {$\gamma_3$};
        \end{scope}
        
        \begin{scope}[yshift=-1.6cm]
            \node[anchor=west] at (.4,-.25) {$c_{v_4}$};
            \node[circle, draw, fill=green!20, minimum size=0.5cm] (p4v3) at (0, 0) {$v_3$};
            \node[circle, draw, fill=green!20, minimum size=0.5cm] (p4v4) at (1.2, 0.3) {$v_4$};
            \draw[->, very thick, green!70!black] (p4v4) -- (p4v3);
            \draw[->, very thick, green!70!black] (p4v3) to[bend left=25] (p4v4);
            
            \node[circle, draw, fill=blue!20, minimum size=0.05] (p4v1) at (2.8, 0.3) {\tiny $v_1$};
            \node[circle, draw, fill=blue!20, minimum size=0.05] (p4v2) at (3.8, -0.1) {\tiny $v_2$};
            \node[circle, draw, fill=orange!20, minimum size=0.5cm] (p4v5) at (4.7, 0.2) {\tiny $v_5$};
            \draw[->, thick, blue] (p4v1) -- (p4v2);
            \draw[->, thick, blue] (p4v2) to[bend left=20] (p4v1);
            \draw[->, very thick, orange] (p4v5) to[loop right, looseness=8] (p4v5);
            \node[blue, right] at (4.2, -.4) {$\gamma_4$};
        \end{scope}
        
        \begin{scope}[yshift=-2.8cm]
            \node[anchor=west] at (.4,-.5) {$c_{v_5}$};
            \node[circle, draw, fill=orange!20, minimum size=0.5cm] (p5v5) at (0.5, 0) {$v_5$};
            \draw[->, very thick, orange] (p5v5) to[loop right, looseness=5] (p5v5);
            
            \node[circle, draw, fill=blue!20, minimum size=0.05] (p5v1) at (2.8, 0.3) {\tiny $v_1$};
            \node[circle, draw, fill=blue!20, minimum size=0.05] (p5v2) at (3.8, 0.15) {\tiny $v_2$};
            \node[circle, draw, fill=green!20, minimum size=0.5cm] (p5v3) at (2.8, -0.3) {\tiny $v_3$};
            \node[circle, draw, fill=green!20, minimum size=0.5cm] (p5v4) at (3.8, -0.5) {\tiny $v_4$};
            \draw[->, thick, blue] (p5v1) -- (p5v2);
            \draw[->, thick, blue] (p5v2) to[bend left=20] (p5v1);
            \draw[->, thick, blue] (p5v3) -- (p5v4);
            \draw[->, thick, blue] (p5v4) to[bend left=20] (p5v3);
            \node[blue, right] at (3.2, -.8) {$\gamma_5$};
        \end{scope}
    \end{scope}
\end{tikzpicture}
\caption{A linear subdigraph $\dot{\gamma}$ with 5 vertices decomposes into exactly 5 GOOD pairs.}
\label{fig:lsd-with-GOOD-pairs}
\end{figure}
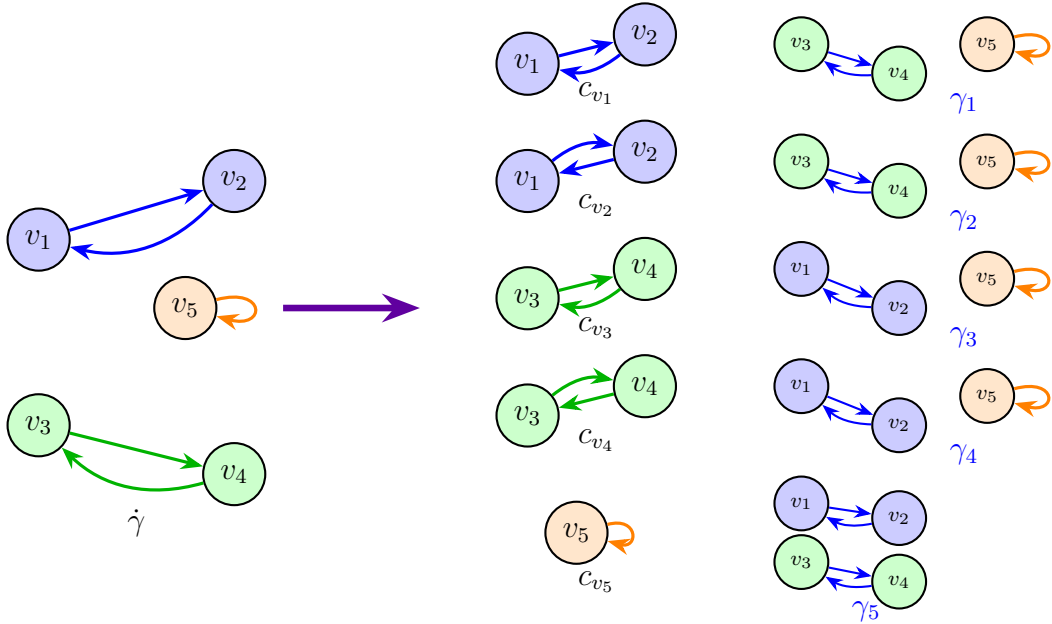
Moreover, for each linear subdigraph $\dot{\gamma}$ on $r$ vertices, the total weight from these $r$ GOOD pairs is
\[
\sum_{i=1}^{r} W((c_{v_i}, \gamma_i)) = \sum_{i=1}^{r} (-1)^{c(\gamma_i)} w(c_{v_i}) w(\gamma_i).
\]
Since $c(\gamma_i) = c(\dot{\gamma}) - 1$ (removing one cycle from $\dot{\gamma}$) and $w(c_{v_i}) w(\gamma_i) = w(\dot{\gamma}),$ we have
\[
\sum_{i=1}^{r} W((c_{v_i}, \gamma_i)) = r \cdot (-1)^{c(\dot{\gamma}) - 1} w(\dot{\gamma}).
\]
Now, the contribution from $r\ell_r$ term is $r \cdot (-1)^{c(\dot{\gamma})} w(\dot{\gamma}).$ Therefore, 
\[
r(-1)^{c(\dot{\gamma})-1}w(\dot{\gamma}) + r(-1)^{c(\dot{\gamma})}w(\dot{\gamma}) = r(-1)^{c(\dot{\gamma})}[(-1)^{-1} + 1] w(\dot{\gamma}) = 0
\]
Consequently, $S = 0$. This completes the proof.

\end{proof}
\begin{proof}[Proof of Theorem \ref{Thm: Trace cayley Hamilton Thm}]
Proof of this theorem follows from Lemmas \ref{Lemma:Coeeffiient of Cha poly of A is sign sum of LSD}, \ref{Lemma:Tr(Ak)=Ck} and Theorem \ref{Thm:Combinatorial proof of Trace Cayley Hamilton thm}.    
\end{proof}

\textbf{Acknowledgement:} The author was supported by the NBHM research project grant
 02011/29/2025 NBHM(R. P)/R\&D II/11951. The author would like to thank Sajal Kumar Mukherjee for all the discussions.
\vspace{3mm}

\textbf{Data Availability Statement:} Availability of data and materials is not
 applicable.

 \vspace{3mm}
 \textbf{Conflict of interest:} The authors declare that they have no conflict of
 interest

\bibliographystyle{amsplain}
\bibliography{gen-inv-lcp-t}

\providecommand{\bysame}{\leavevmode\hbox to3em{\hrulefill}\thinspace}
\providecommand{\MR}{\relax\ifhmode\unskip\space\fi MR }
\providecommand{\MRhref}[2]{%
  \href{http://www.ams.org/mathscinet-getitem?mr=#1}{#2}
}
\providecommand{\href}[2]{#2}
\begin{thebibliography}{1}

\bibitem{Brualdi-Cvetkovic-combinatorial-matrix-thy}
A.~R. Brualdi and D.~Cvetkovic, \emph{A combinatorial approach to matrix theory
  and its application}, Discrete Mathematics and Its Applications, vol.~52, CRC
  Press, Boca Raton, London, New York, 2009.

\bibitem{Darij-Grinberg-Trace-Cayley-Hamilton}
D.~Grinberg, \emph{The trace {C}ayley-{H}amilton theorem},  (2019).

\bibitem{Comb-Newton-Girad-Disc_Mukherjee-Bera}
S.~K. Mukherjee and S.~Bera, \emph{Combinatorial proofs of the
  {N}ewton–{G}irard and {C}hapman–{C}ostas-{S}antos identities}, Discrete
  Math. \textbf{342} (2019), 1577--1580.

\bibitem{Straubing1983}
H.~Straubing, \emph{A combinatorial proof of the {C}ayley-{H}amilton theorem},
  Discrete Math. \textbf{43} (1983), no.~2--3, 273--279.

\bibitem{11}
D.~Zeilberger, \emph{A combinatorial proof of {N}ewton's identity}, Discrete
  Math. \textbf{49} (1984), 319.

\bibitem{22}
\bysame, \emph{A combinatorial approach to matrix algebra}, Discrete Math.
  \textbf{56} (1985), 61--72.

\end{thebibliography}

\end{document}